\documentclass[a4paper,11pt]{article}
\usepackage[latin1]{inputenc}
\usepackage[english]{babel}
\usepackage{amsmath}
\usepackage{amsfonts}
\usepackage{amssymb}
\usepackage{epsfig}
\usepackage{amsopn}
\usepackage{amsthm}
\usepackage{color}
\usepackage{graphicx}
\usepackage{subfigure}
\usepackage{enumerate}
\newtheorem{theorem}{Theorem}[section]
\newtheorem{corollary}[theorem]{Corollary}

\newtheorem{proposition}[theorem]{Proposition}
\newtheorem{definition}[theorem]{Definition}

\newtheorem*{theorem*}{Theorem}
\newtheorem*{lemma*}{Lemma}
\newtheorem*{remark*}{Remark}
\newtheorem*{definition*}{Definition}
\newtheorem*{proposition*}{Proposition}
\newtheorem*{corollary*}{Corollary}
\numberwithin{equation}{section}
\newcommand{\real}{\mathbb{R}}

\def\qed{\,\unskip\kern 6pt \penalty 500
\raise -2pt\hbox{\vrule \vbox to8pt{\hrule width 6pt
\vfill\hrule}\vrule}\par}
\definecolor{darkblue}{rgb}{0.05, .05, .65}
\definecolor{darkgreen}{rgb}{0.1, .65, .1}
\definecolor{darkred}{rgb}{0.8,0,0}
\newcommand{\beqn}{\begin{equation}}
\newcommand{\eeqn}{\end{equation}}
\newcommand{\bear}{\begin{eqnarray}}
\newcommand{\eear}{\end{eqnarray}}
\newcommand{\bean}{\begin{eqnarray*}}
\newcommand{\eean}{\end{eqnarray*}}
\begin{document}
%%%%%%%%%%%%%%%%%%%%%%%%%%%%%%%%%%%%%%%%%%%%%%%%%

%%%%%%%%%%%%%%%%%%%%%%%%%%%%%%%%%%%%%%%%%%%%%%%%%
\title{\huge \bf Some qualitative properties of solutions to a reaction-diffusion equation with weighted strong reaction}

\author{
\Large Razvan Gabriel Iagar\,\footnote{Departamento de Matem\'{a}tica
Aplicada, Ciencia e Ingenieria de los Materiales y Tecnologia
Electr\'onica, Universidad Rey Juan Carlos, M\'{o}stoles,
28933, Madrid, Spain, \textit{e-mail:} razvan.iagar@urjc.es},\\
[4pt] \Large Ana I. Mu\~{n}oz\,\footnote{Departamento de Matem\'{a}tica
Aplicada, Ciencia e Ingenieria de los Materiales y Tecnologia
Electr\'onica, Universidad Rey Juan Carlos, M\'{o}stoles,
28933, Madrid, Spain, \textit{e-mail:} anaisabel.munoz@urjc.es},
\\[4pt] \Large Ariel S\'{a}nchez,\footnote{Departamento de Matem\'{a}tica
Aplicada, Ciencia e Ingenieria de los Materiales y Tecnologia
Electr\'onica, Universidad Rey Juan Carlos, M\'{o}stoles,
28933, Madrid, Spain, \textit{e-mail:} ariel.sanchez@urjc.es}\\
[4pt] }
\date{}
\maketitle

\begin{abstract}
We study the existence and qualitative properties of solutions to the Cauchy problem associated to the quasilinear reaction-diffusion equation
$$
\partial_tu=\Delta u^m+(1+|x|)^{\sigma}u^p,
$$
posed for $(x,t)\in\real^N\times(0,\infty)$, where $m>1$, $p\in(0,1)$ and $\sigma>0$. Initial data are taken to be bounded, non-negative and compactly supported. In the range when $m+p\geq2$, we prove \emph{local existence of solutions} together with a \emph{finite speed of propagation} of their supports for compactly supported initial conditions. We also show in this case that, for a given compactly supported initial condition, there exist \emph{infinitely many solutions} to the Cauchy problem, by prescribing the evolution of their interface. In the complementary range $m+p<2$, we obtain new \emph{Aronson-B\'enilan estimates} satisfied by solutions to the Cauchy problem, which are of independent interest as a priori bounds for the solutions. We apply these estimates to establish \emph{infinite speed of propagation} of the supports of solutions if $m+p<2$, that is, $u(x,t)>0$ for any $x\in\real^N$, $t>0$, even in the case when the initial condition $u_0$ was compactly supported.
\end{abstract}

\


\smallskip


\section{Introduction}

This paper deals with the qualitative theory of the weak solutions to the Cauchy problem for the following reaction-diffusion equation
\begin{equation}\label{eq1}
\partial_tu=\Delta u^m+(1+|x|)^{\sigma}u^p, \qquad  (x,t)\in \real^N\times(0,\infty), \ N\geq1,
\end{equation}
supplemented with the initial condition
\begin{equation}\label{eq2}
u(x,0)=u_0(x), \quad x\in\real^N.
\end{equation}
The exponents in \eqref{eq1} are considered throughout the paper to belong to the following range
\begin{equation}\label{exp}
m>1, \qquad 0<p<1, \qquad 0<\sigma<\infty,
\end{equation}
although we also give alternative proofs or even slight improvements of known results with $\sigma=0$. We will work in general with bounded, compactly supported, non-negative and non-trivial initial conditions, more precisely
\begin{equation}\label{icond}
u_0\in L^{\infty}(\real^N), \qquad {\rm supp}\,u_0\subseteq B(0,R), \qquad u_0(x)\geq0, \ {\rm for \ any} \ x\in\real^N, \ u_0\not\equiv0,
\end{equation}
and further regularity assumptions (such as continuity) will be specified at the points where they are needed. In nonlinear diffusion problems it is standard to consider data belonging to the space $L^1_{\rm loc}(\real^N)$ and we are sure that some of our results can be extended to this weaker space than $L^{\infty}(\real^N)$. However, for simplicity and also as in the range of exponents \eqref{exp} finite time blow-up of bounded solutions is expected (as established, for example, for self-similar solutions in recent works such as \cite{IS20a, IS20b, IMS22}), we decided to avoid possible pointwise singularities at finite points and thus require \eqref{icond}.

A deep study of Eq. \eqref{eq1} in the range $p>1$ has been performed by Andreucci and DiBenedetto in \cite{AdB91} for weights with any $\sigma\in\real$, that is, both positive and negative. Properties such as local existence of weak solutions under optimal growth conditions on the initial data, estimates, regularity of them and Harnack-type inequalities are obtained. The authors also specify in \cite{AdB91} that some of their results can be extended to the range $0<p<1$ but only when $\sigma<0$ and leave open the case $\sigma>0$ with $0<p<1$. We find this latter case very interesting to study due to the merging of the non-Lipschitz reaction term (which does not produce finite time blow-up by itself) with the unbounded weight.

The non-weighted equation
\begin{equation}\label{eq4}
u_t=\Delta u^m+u^p,
\end{equation}
corresponding to $\sigma=0$ in Eq. \eqref{eq1} is nowadays quite well understood (at least in dimension $N=1$) after a series of works by de Pablo and V\'azquez \cite{dPV90, dPV91, dPV92, dP94} and the outcome of them is very interesting, despite the fact that we are dealing with an ill-posed problem. It is shown that solutions \emph{exist always} and they are global in time if $u_0$ satisfies a growth condition similar to the one required for existence in the porous medium equation, namely $u_0(x)=o(|x|^{2/(m-1)})$ \cite{dPV91} and that there is always a minimal and a maximal solution, both constructed via approximations. The most interesting and surprising features of this problem with $\sigma=0$ are related with the uniqueness. This property depends strongly on two aspects: the sign of the critical exponent $m+p-2$ and the positivity or not of the initial condition $u_0$. More precisely

$\bullet$ if $m+p-2\geq0$, uniqueness of solutions to the Cauchy problem \eqref{eq4}-\eqref{eq2} holds true if and only if ${\rm supp}\,u_0(x)=\real^N$ \cite{dPV90}. If $u_0$ is compactly supported, then its support has the property of finite propagation and interfaces appear. Moreover, the extent of the non-uniqueness property in this case is addressed in \cite{dPV92} where it is shown that giving a rather general compactly supported initial condition $u_0$ and a function of time $\xi(t)$ advancing faster than the interface of the (unique) minimal solution constructed in \cite{dPV91}, there is always a solution to the Cauchy problem with data $u_0$ and interface at time $t>0$ given by $\xi(t)$. This is a very strong and sharp non-uniqueness property, giving rise in fact to an infinity of solutions.

$\bullet$ if $m+p-2<0$ things change radically due to the infinite speed of propagation: even if $u_0$ is compactly supported it is shown in \cite{dPV90} that $u(x,t)>0$ for any $x\in\real^N$ and $t>0$, thus we have a property known as \emph{quasi-uniqueness}: there exists a unique solution to the Cauchy problem for any initial condition $u_0$ except for the trivial one $u_0\equiv0$, where two different solutions are constructed.

Considering weighted reaction terms came as a natural extension of the already well developed knowledge on the ``classical" reaction-diffusion equations with reaction of the form $u^p$ or more general functions resembling it (see \cite{QS,S4} as important monographs on this subject). Many results were achieved for the \emph{semilinear case $m=1$} and unbounded weights of the form $|x|^{\sigma}u^p$ (and sometimes even more general weights $V(x)$ instead of pure powers), always with $p>1$. Fujita-type exponents and conditions on the data for finite-time blow-up to occur were studied in celebrated papers by Baras and Kersner, Bandle and Levine, Pinsky et. al., see for example \cite{BK87, BL89, Pi97, Pi98}. More recently, still with $m=1$, an interesting question was addressed: considering the equation
\begin{equation}\label{eq3.bis}
u_t=\Delta u+|x|^{\sigma}u^p,
\end{equation}
it is natural to ask ourselves whether $x=0$ (or more generally, any zero of a power-like weight $V(x)$) can be a blow-up point. Examples of both possible situations (when $x=0$ is a blow-up point and when it is not) were constructed (mostly for Cauchy-Dirichlet problems posed in bounded domains) in the series of recent papers by Guo and collaborators \cite{GLS10, GLS13, GS11, GS18}. Finer analysis on how blow-up occurs, with rates and local asymptotic behavior in self-similar forms near the blow-up time and points was performed by Filippas and Tertikas \cite{FT00} and in the very recent work by Mukai and Seki \cite{MS20} for different ranges of the exponent $p>1$.

Due to their further complexity and the fact that even for the non-weighted case $\sigma=0$ there are some difficult open problems (see for example \cite[Chapter 4]{S4}), equations such as \eqref{eq1} or its close relative
\begin{equation}\label{eq3}
u_t=\Delta u^m+|x|^{\sigma}u^p,
\end{equation}
with $m>1$ have been less considered in literature. Apart from the quoted paper \cite{AdB91}, the Fujita-type exponent and rather sharp conditions on the initial data for the finite time blow-up to hold true have been obtained by Qi \cite{Qi98} and then Suzuki \cite{Su02} for the case $p>m>1$ (including also a part of the fast diffusion range $m<1$ in \cite{Qi98}). We recommend Suzuki's paper to the reader as a well-written basic work on the qualitative theory for these equations, while blow-up rates as $t\to T$ also for $p>m$ are proved in \cite{AT05}.

In recent years, the authors of the present work started a larger project of understanding the patterns (in self-similar form) that solutions may take either in the case of global solutions, or close to the blow-up time if this occurs, for Eq. \eqref{eq3}. This is an important part of the study, since it is well-known that such patterns are a prototype of the general behavior of the equation, in form of asymptotic profiles as $t\to\infty$ or $t\to T$ and also bring a deeper understanding on the blow-up sets and rates. A series of papers \cite{IS19a, IS21, IS22, ILS23} address the question of the blow-up profiles to Eq. \eqref{eq3} for $m>1$ and $1\leq p\leq m$, where interesting and rather unexpected behaviors were established. In all these cases, solutions blow up in finite time when $\sigma>0$, but their specific blow-up behavior it is shown to depend strongly on the magnitude of $\sigma$. Going back to the case of interest for us, $p\in(0,1)$, we have proved that blow-up profiles exist if the following condition is fulfilled:
\begin{equation}\label{const}
L:=\sigma(m-1)+2(p-1)>0.
\end{equation}
We classified such blow-up profiles in \cite{IMS22} in general dimensions $N\geq2$, following previous results restricted to dimension $N=1$ in \cite{IS20b, IS20c}, obtaining again that the sign of $m+p-2$ is fundamental for their existence and behavior. More precisely,

$\bullet$ when $m+p-2>0$ all the blow-up self-similar profiles present an interface (that is, they are compactly supported) and there are two different types of possible interface behaviors. This is both a manifestation of non-uniqueness (expected for compactly supported data) and a suggestion of possible non-existence of solutions when $u_0(x)>0$, as there is no pattern they can approach when $t\to T$.

$\bullet$ when $m+p=2$ self-similar blow-up patterns present a rather similar panorama, but with a single type of interface behavior.

$\bullet$ when $m+p<2$ a rather striking non-existence of any kind of blow-up profiles (either with interfaces or not) occurs \cite{IS20b}. Such an outcome gives the intuition of an infinite speed of propagation and a complete non-existence of non-trivial solutions.

In this paper we thus begin the qualitative study of the Cauchy problem for Eq. \eqref{eq1} posed in $\real^N$. We thus address the quite interesting (in view of precedents such as \cite{dPV90, dPV92}) question of speed of propagation and non-uniqueness for compactly supported data, deriving in the process some new Aronson-B\'enilan estimates for the solutions to Eq. \eqref{eq1} when $m+p<2$. We stress here that uniform lower bounds on the weight are essential in the forthcoming proofs, thus technical complications introduced by the weight in a neighborhood of $x=0$ (where it is no longer uniformly positive) appear when trying to adapt the same proofs to the close relative Eq. \eqref{eq3}.

But let us present below in more detail our main results.

\medskip

\noindent \textbf{Main results.} In order to state our results concerning the qualitative theory of solutions to Eq. \eqref{eq1}, we first have to introduce the notion of \emph{weak solution} that will be used throughout the paper. Let us denote by $u(t)$ the mapping $x\mapsto u(x,t)$ for $t>0$ fixed. We will slightly modify the functional framework from Andreucci and DiBenedetto \cite{AdB91} by passing in our weak formulation the full Laplacian to the test function.
\begin{definition}\label{def.sol}
A non-negative function $u:\real^N\times(0,T)\mapsto[0,\infty)$ is said to be a \emph{weak solution} to the Cauchy problem \eqref{eq1}-\eqref{eq2} with initial condition $u_0$ as in \eqref{icond} if the following assumptions are satisfied by $u$:

(a) \emph{Regularity assumption}: we have
\begin{equation}\label{reg.sol}
u\in C(0,T;L^1_{\rm loc}(\real^N))\cap L^{\infty}_{\rm loc}(\real^N\times(0,T)).
\end{equation}

(b) \emph{Weak formulation of \eqref{eq1}}: for any test function $\eta\in C_0^{\infty}(\real^N\times(0,T))$ and for any $t\in(0,T)$ we have
\begin{equation}\label{weak.sol}
\begin{split}
\int_{\real^N}u(x,t)\eta(x,t)\,dx&+\int_0^t\int_{\real^N}\left(-u(x,\tau)\eta_t(x,\tau)-u^m(x,\tau)\Delta\eta(x,\tau)\right)dx\,d\tau\\&=
\int_0^t\int_{\real^N}(1+|x|)^{\sigma}u^p(x,\tau)\eta(x,\tau)dx\,d\tau.
\end{split}
\end{equation}

(c) \emph{Taking the initial condition}: this is done in $L^1$ sense, more precisely
\begin{equation}\label{limit.sol}
\lim\limits_{t\to0}u(t)=u_0, \qquad {\rm with \ convergence \ in} \ L^1_{\rm loc}(\real^N)
\end{equation}
\end{definition}
This change relaxes the functional assumptions of regularity for a solution, making it easier to obtain weak solutions by limiting processes. We will also need throughout the paper the notions of (weak) sub- and supersolution to Eq. \eqref{eq1}. We say that $u$ is a \emph{weak subsolution} (respectively \emph{weak supersolution}) to Eq. \eqref{eq1} if condition (a) is fulfilled and condition (b) is modified in the sense that the equal sign is replaced by $\leq$ (respectively $\geq$) for any test functions $\eta\in C_0^{\infty}(\real^N\times(0,T))$ such that $\eta\geq0$ and for any $t\in(0,T)$. Moreover, the notions of weak solution, subsolution, supersolution to Eq. \eqref{eq1} (without the initial condition) can be defined in an obvious way on time intervals $[t_1,t_2]\subset(0,T)$ instead of $(0,T)$ by just removing assumption (c).

Once defined the notion of weak solution, we are in a position to give below the main theorems of this work. As explained above, it is expected from the results in papers such as \cite{dPV90, dPV91, dPV92} for equations with non-weighted reaction that in our range of exponents ill-posedness of the Cauchy problem will still hold true and both existence and uniqueness of solutions are an issue. As we shall see below, if we consider initial conditions $u_0$ as in \eqref{icond}, these basic properties of existence and uniqueness of weak solutions are strongly related to the finite or infinite speed of propagation of their edge of the support. Throughout the paper, we will denote by $C_0(\real^N)$ the space of continuous, compactly supported functions on $\real^N$.

\medskip

\noindent \textbf{The range} $\mathbf{m+p\geq2}$. Let us first focus on the range of exponents for which $m+p\geq2$ and $\sigma>0$. In this case, we shall prove that, at least for some interval of time $t\in(0,T)$, there exist infinitely many weak solutions to Eq. \eqref{eq1}. We begin with the existence of at least one local (in time) weak solution, as stated in the following:
\begin{theorem}[Local existence of compactly supported solutions]\label{th.exist}
In our framework and notation, assume that $m+p\geq2$ and let $u_0$ be an initial condition as in \eqref{icond} satisfying moreover that $u_0\in C_0(\real^N)$. Then there exists $T>0$ and there exists at least a weak solution $u$ to the Cauchy problem \eqref{eq1}-\eqref{eq2} for $t\in(0,T)$ which remains continuous and compactly supported: $u(t)\in C_0(\real^N)$ for any $t\in(0,T)$.
\end{theorem}
The proof relies on the construction of so-called \emph{minimal solutions} to the Cauchy problem \eqref{eq1}-\eqref{eq2}, which are obtained through a limit process from a family of approximating (regularized) Cauchy problems. It will be shown that such a minimal solution exists for any compactly supported condition $u_0$ and stays compactly supported for $t\in(0,T)$, a property known as \emph{finite speed of propagation}. The (local in time) finite speed of propagation will be proved with the aid of comparison with solutions and supersolutions in self-similar form introduced in the recent works \cite{IS20b, IS20c, IMS22}.

The statement of Theorem \ref{th.exist} and precedents in the non-weighted case \cite{dPV92} give the idea that uniqueness does not hold true. In fact, we infer from Theorem \ref{th.exist} that at intuitive level the weight $V(x)=(1+|x|)^{\sigma}$ is equivalent to a constant for times $t\in(0,T)$ (that is, while the support of $u(t)$ remains finite), thus a similar property to the non-weighted case concerning non-uniqueness of solutions is expected. The next result characterizes the extent of this non-uniqueness, by showing that we can prescribe in infinitely many ways the evolution of the interface of a solution stemming from the same initial condition.
\begin{theorem}[Non-uniqueness of compactly supported solutions]\label{th.nonuniq}
In our framework and notation, assume that $m+p\geq2$ and let $u_0$ be an initial condition as in \eqref{icond} satisfying moreover that $u_0\in C_0(\real)$. Then there exists $T>0$ and infinitely many weak solutions to the Cauchy problem \eqref{eq1}-\eqref{eq2} for $t\in(0,T)$. The same existence of infinitely many weak solutions holds true for radially symmetric initial conditions $u_0\in C_0(\real^N)$ satisfying \eqref{icond}.
\end{theorem}
This result is rather similar to the non-weighted case $\sigma=0$, although the proof will be technically more involved: it is not clear whether there exists with $\sigma>0$ a maximal solution, in order to get almost for free a different, second solution (as it holds true for $\sigma=0$, see \cite{dPV90, dPV91}), thus we have to use a different approach. The statement of Theorem \ref{th.nonuniq} will be thus enforced and made more precise in Section \ref{sec.nonuniq}, where we show that the existence of infinitely many solutions is linked to a prescribed evolution of their interface in time, adapting but also slightly improving techniques from \cite{dPV92}.

\medskip

\noindent \textbf{The range} $\mathbf{m+p<2}$. In this complementary range, our main goal is to prove that solutions to the Cauchy problem for Eq. \eqref{eq1} with initial conditions as in \eqref{icond} have \emph{infinite speed of propagation}, that is, they become immediately positive at any point $x\in\real^N$. This fact extends to the non-homogeneous range $\sigma>0$ a similar result holding true for $\sigma=0$, established in \cite{dPV90}, but we use a different approach which gives, along the way, a result of independent interest in the study of the homogeneous equation \eqref{eq4}. We begin by establishing the following \emph{Aronson-B\'enilan estimates} (whose name stems from their celebrated short note \cite{ABE79} on solutions to the porous medium equation) for solutions to the homogeneous case $\sigma=0$. We thus introduce the \emph{pressure function}
$$
v=\frac{m}{m-1}u^{m-1}.
$$
\begin{theorem}[Aronson-B\'enilan estimates when $m+p<2$ and $\sigma=0$]\label{th.ABE}
Assume that $m+p<2$. Let $u$ be a weak solution to Eq. \eqref{eq4} in $\real^N\times(0,T)$ with a continuous initial condition $u_0(x)=u(x,0)$ satisfying \eqref{icond} and let $v$ be the pressure variable introduced above. Then the following inequality
\begin{equation}\label{ABE}
\Delta v\geq-\frac{K}{t}, \qquad K=\frac{N}{N(m-1)+2},
\end{equation}
holds true in the sense of distributions in $\real^N$, that means,
\begin{equation}\label{ABEdist}
\int_0^{T}\int_{\real^N}\left(v(x,t)\Delta\varphi(x,t)+\frac{K}{t}\varphi(x,t)\right)\,dx\,dt\geq0,
\end{equation}
for any test function $\varphi\in C_0^{\infty}(\real^N\times(0,T))$ such that $\varphi\geq0$ in $\real^N\times(0,T)$.
\end{theorem}
\noindent \textbf{Remark.} Theorem \ref{th.ABE} is expected to hold true also \emph{for any $\sigma>0$}, provided $N\geq2$. We give a formal proof of this fact at the end of Section \ref{sec.ABE}. However, transforming it into a rigorous proof is impossible by now for $\sigma>0$, as we are lacking a well-posedness result for the Cauchy problem \eqref{eq1}-\eqref{eq2} in the range $m+p<2$. This is why, we introduce this formal proof as a remark.

We then employ the Aronson-B\'enilan estimates established in Theorem \ref{th.ABE} and some consequences of them in order to prove the infinite speed of propagation of the supports of solutions to Eq. \eqref{eq1} when $m+p-2<0$, which is strongly contrasting to the results established in the range $m+p\geq2$. We include the range $\sigma=0$ in the statement, as our approach gives an alternative proof to the one in \cite{dPV90}.
\begin{theorem}[Infinite speed of propagation when $m+p<2$]\label{th.ISP}
If the exponents $m$ and $p$ satisfy $m+p<2$, Eq. \eqref{eq1} has the property of infinite speed of propagation for any $\sigma\geq0$, that is, every weak solution (if it exists) to the Cauchy problem associated to Eq. \eqref{eq1} with continuous initial condition $u_0$ as in \eqref{icond} satisfies $u(x,t)>0$ for any $x\in\real^N$ and $t>0$.
\end{theorem}
The rest of the paper is devoted to the proofs of the main theorems, following the outlines explained in the comments near their statements. We then give at the end of the paper a list of open problems and possible extensions of our study that we believe to be interesting for future developments.

\section{Existence and finite speed of propagation when $m+p\geq2$}\label{sec.localex}

The goal of this section is to prove Theorem \ref{th.exist} in the range of parameters $m+p\geq2$. Let $u_0$ be a continuous initial condition as in \eqref{icond}. The scheme of the proof is based on the construction of a \emph{minimal solution} via an approximation process and showing that this minimal solution is a weak solution to the Cauchy problem \eqref{eq1}-\eqref{eq2} with compact support at any time $t\in(0,T)$ for some $T>0$.
\begin{proposition}\label{prop.minimal}
There exists some $T>0$ and a continuous weak solution $\underline{u}$ defined and compactly supported for $t\in(0,T)$ to the Cauchy problem \eqref{eq1}-\eqref{eq2} with $u_0$ as above such that for any other weak solution $u$ to the Cauchy problem \eqref{eq1}-\eqref{eq2} (if it exists), we have
$$
\underline{u}(x,t)\leq u(x,t), \qquad {\rm for \ any} \ (x,t)\in\real^N\times(0,T).
$$
\end{proposition}
\begin{proof}
We divide the proof into three steps for the reader's convenience. Let us stress at this point that, while Step 1 of the proof is an adaptation of an analogous construction in \cite{dPV91}, the idea in Steps 2 and 3 strongly departs from the one used in the previously mentioned work and employs results on self-similar solutions to Eq. \eqref{eq3} published recently by the authors.

\medskip

\noindent \textbf{Step 1. The construction of the minimal solution.} We approximate the Cauchy problem \eqref{eq1}-\eqref{eq2} by the following sequence of Cauchy problems for any positive integer $k\geq1$
\begin{equation}\label{approx.min}
(P_k) \ \left\{\begin{array}{ll}w_t=\Delta w^m+\min\{(1+|x|)^{\sigma},k\}f_k(w), & (x,t)\in\real^N\times(0,\infty),\\
w(x,0)=u_0(x), & x\in\real^N\end{array}\right.
\end{equation}
where
$$
f_k(w)=\left\{\begin{array}{ll}\left(\frac{1}{k}\right)^{p-1}w, & {\rm if} \ 0\leq w\leq\frac{1}{k},\\
w^p, & {\rm if} \ w\geq\frac{1}{k}\end{array}\right.
$$
Since the nonlinearity in \eqref{approx.min} is of the form $g(x)h(w)$ with $g\in L^{\infty}(\real^N)$, $g(x)\geq 1$ for any $x\in\real^N$ and $h$ is a Lipschitz function, we infer by standard results for quasilinear parabolic equations (see for example \cite{DiB83, Sa83}) that the Cauchy problem $(P_k)$ admits a unique solution $w_k$ defined for $(x,t)\in\real^N\times(0,\infty)$, which is compactly supported and continuous. The comparison principle is in force for the Cauchy problem $(P_k)$ and $w_{k+1}$ is a supersolution to the problem $(P_k)$, thus $w_{k+1}\geq w_k$ for any $k\geq1$. This allows us to introduce the pointwise (and monotone increasing) limit (which might become infinite starting from some finite time)
$$
\underline{u}(x,t)=\lim\limits_{k\to\infty}w_k(x,t)<\infty, \qquad (x,t)\in\real^N\times(0,T_{\infty}),
$$
which is well defined provided that $T_{\infty}>0$. This fact will follow from the construction of a ``universal" family of supersolutions in self-similar form which is postponed to Step 2 (in dimension $N=1$) and Step 3 (in dimension $N\geq2$) below. Moreover, the solutions $w_k$ are thus uniformly bounded on $\real^N\times(0,T)$ for some $T=T(u_0)>0$ depending on $u_0$, and this uniform boundedness, together with classical results in \cite{DiB83, Sa83}, imply that the family $(w_k)_{k\geq1}$ is uniformly equicontinuous in $\real^N\times[0,T]$, hence there exists a subsequence (relabeled also $w_k$ for simplicity) which converges locally uniformly to the same function $\underline{u}(x,t)$.

Then the fact that $\underline{u}$ is a continuous weak solution to \eqref{eq1}-\eqref{eq2} for $(x,t)\in\real^N\times(0,T_{\infty})$ follows now readily from the previous convergences (assuming for now the outcome of Steps 2 and 3 below): indeed, Lebesgue's monotone convergence theorem ensures the assumptions (b) and (c) in Definition \ref{def.sol}, while the uniform bound by the supersolutions constructed in Steps 2 and 3 below, together with the equicontinuity, give that $\underline{u}$ satisfies the regularity assumption (a) in Definition \ref{def.sol}. Moreover, if $u$ is another weak solution to the Cauchy problem \eqref{eq1}-\eqref{eq2}, then it is a supersolution to the problem $(P_k)$ for any $k\geq1$, whence $u(x,t)\geq w_k(x,t)$ for any $k\geq1$ and $(x,t)\in\real^N\times(0,\infty)$ and by passing to the limit $u\geq\underline{u}$ in $\real^N\times(0,T_{\infty})$, proving the minimality of $\underline{u}$.

\medskip

\noindent \textbf{Step 2. Supersolutions in dimensions $N=1$.} We are left with the task of obtaining a uniform bound from above for all solutions $w_k$ to the Cauchy problems $(P_k)$, $k\geq1$, at least up to some (short) finite time. This follows by comparison with suitable supersolutions in self-similar form constructed in recent works by the authors such as \cite{IMS22} for $N\geq2$ and $m+p\geq2$, respectively \cite{IS20b, IS20c} in dimension $N=1$ and either $m+p>2$ or $m+p=2$. If we restrict ourselves only to dimension $N=1$ for this step, it is shown in the previously quoted works that, in our range of exponents together with the extra condition $\sigma>2(1-p)/(m-1)$, there are radially symmetric blow-up self-similar supersolutions to Eq. \eqref{eq3} in the form
\begin{equation}\label{SSS}
u(x,t)=(T-t)^{-\alpha}f(|x|(T-t)^{\beta}), \ \alpha=\frac{\sigma+2}{L}, \ \beta=\frac{m-p}{L}
\end{equation}
where $L>0$ has been defined in \eqref{const}, such that their self-similar profiles $f$ solve the differential equation
\begin{equation}\label{SSODE}
(f^m)''(\xi)-\alpha f(\xi)+\beta\xi f'(\xi)+\xi^{\sigma}f(\xi)^p=0, \qquad \xi=|x|^{\sigma}(T-t)^{\beta}.
\end{equation}
Moreover, it is established in \cite[Proposition 4.1]{IS20b} if $m+p>2$ and in \cite[Proposition 3.1]{IS20c} if $m+p=2$ that the self-similar profiles $f$ of the supersolutions in the form \eqref{SSS} fulfill the following two additional properties:

$\bullet$ $f$ is strictly decreasing until reaching the zero level: $f(0)=A>0$, $f'(\xi)<0$ at points $\xi\geq0$ where $f(\xi)>0$.

$\bullet$ $f$ presents an interface at some finite point $\xi_0\in(0,\infty)$, that is, $f(\xi_0)=0$, $f(\xi)>0$ for any $\xi\in(0,\xi_0)$ and $(f^m)'(\xi_0)=0$.

Here the blow-up time $T>0$ is a free parameter and the functions defined in \eqref{SSS} are actually weak solutions to Eq. \eqref{eq3} except at the point $x=0$ where the condition $(f^m)'(0)=0$ is not fulfilled in order to be a weak solution. We adapt these supersolutions to our Eq. \eqref{eq1} by defining
\begin{equation}\label{supersol}
z(x,t)=(T-t)^{-\alpha}f((1+|x|)(T-t)^{\beta}),
\end{equation}
with $\alpha$, $\beta$ and $f$ as in \eqref{SSS}. Since $f$ is a supersolution to the differential equation \eqref{SSODE}, it is straightforward to check that $z$ is a supersolution to \eqref{eq1}. The amplitude $s(t)$ of the support of $z(t)$ at some $t\in(0,T)$ is given by $\xi_0=(1+s(t))(T-t)^{\beta}$ or equivalently
$$
s(t)=(T-t)^{-\beta}\xi_0-1\to\infty, \qquad {\rm as} \ t\to T.
$$
Moreover, the above supersolutions have been established in \cite{IS20b, IS20c} only under the condition $\sigma>2(1-p)/(m-1)$. However, if $0<\sigma\leq2(1-p)/(m-1)$, it is obvious from the fact that $1+|x|\geq1$, that $(1+|x|)^{\sigma}\leq(1+|x|)^{\sigma_1}$ for any $\sigma_1>2(1-p)/(m-1)$ and for any $x\in\real$. It thus follows that the supersolutions constructed for Eq. \eqref{eq1} with such an exponent $\sigma_1>2(1-p)/(m-1)$ as above, will serve also as supersolutions to Eq. \eqref{eq1} with exponents $\sigma$ smaller. Since $T$ is a free parameter and $u_0$ is bounded and compactly supported, we can choose some $T_0>0$ sufficiently small such that
\begin{equation}\label{interm11}
z(x,0)=T_0^{-\alpha}f((1+|x|)T_0^{\beta})\geq\|u_0\|_{\infty}\geq u_0(x)
\end{equation}
for any $x\in{\rm supp}\,u_0$. This, together with the fact that any function $z$ as in \eqref{supersol} is a supersolution to Eq. \eqref{eq1}, gives that $z$ is also a supersolution to the Cauchy problem $(P_k)$ for any $k\geq1$ for which the comparison principle applies to give that
$$
w_k(x,t)\leq z(x,t), \qquad {\rm for \ any} \ (x,t)\in\real\times(0,T_0).
$$
In particular we infer on the one hand that $T_{\infty}\geq T_0>0$ as claimed, and on the other hand by passing to the limit as $k\to\infty$ we get
$$
\underline{u}(x,t)\leq z(x,t), \qquad {\rm for \ any} \ (x,t)\in\real\times(0,T_0)
$$
which proves the finite speed of propagation of the support of $\underline{u}$ at least for some short interval of time.

\medskip

\noindent \textbf{Step 3. Supersolutions in dimension $N\geq2$.} In this case, there are no longer bounded and decreasing supersolutions in the self-similar form used in Step 2. We thus construct suitable supersolutions to Eq. \eqref{eq3} by joining two different self-similar profiles. We begin again by fixing $\sigma>2(1-p)/(m-1)$ as a first case. We consider in general self-similar solutions to Eq. \eqref{eq3} in the same form \eqref{SSS} as above, with the same exponents $\alpha$ and $\beta$, but whose profiles solve the differential equation
\begin{equation}\label{SSODE.N}
(f^m)''(\xi)+\frac{N-1}{\xi}(f^m)'(\xi)-\alpha f(\xi)+\beta\xi f'(\xi)+\xi^{\sigma}f(\xi)^p=0, \qquad \xi=|x|^{\sigma}(T-t)^{\beta}.
\end{equation}
On the one hand, the analysis in \cite[Proposition 4.1]{IMS22} if $m+p>2$, respectively \cite[Proposition 4.2 and Lemma 4.3]{IMS22} if $m+p=2$, ensures that, given $\xi_0\in(0,\xi_*)$ sufficiently small, there exists a decreasing self-similar profile $f_2(\xi)$ solution to the differential equation \eqref{SSODE.N} having an interface exactly at $\xi=\xi_0$ and a vertical asymptote as $\xi\to0$ with local behavior
$$
f_2(\xi)\sim\left\{\begin{array}{ll}C\xi^{-(N-2)/m}, & {\rm if} \ N\geq3,\\ C(-\ln\,\xi)^{1/m}, & {\rm if} \ N=2,\end{array}\right.
$$
as it follows from \cite[Lemma 3.2 and Lemma 3.5]{IMS22}, where $C>0$ designs a positive constant that might change from one case to another. On the other hand, the analysis performed in \cite[Lemma 3.1]{IMS22} implies that, for any $A>0$, there exists a profile $f_1(\xi;A)$ local solution to Eq. \eqref{SSODE.N} such that
$$
f_1(0;A)=A, \qquad f_1(\xi;A)\sim\left[A^{m-1}+\frac{\alpha(m-1)}{2mN}\xi^2\right]^{1/(m-1)}, \qquad {\rm as} \ \xi\to0,
$$
which is increasing in a right-neighborhood of the origin up to some maximum point $\xi_1(A)>0$. Thus, given an initial condition $u_0$ as in \eqref{icond}, one can choose for example $A=\|u_0\|_{\infty}$, fix some $\xi_0\in(0,\xi_1(A))\cap(0,\xi_*)$ such that there exists a decreasing profile $f_2(\xi)$ as above with vertical asymptote as $\xi\to0$ and edge of the support at $\xi=\xi_0$. These two profiles have to cross at some point $\overline{\xi}\in(0,\xi_1(A))$. We finally define a self-similar supersolution to Eq. \eqref{eq1} as follows:
\begin{equation}\label{supersol.N}
z(x,t)=(T-t)^{-\alpha}f((1+|x|)(T-t)^{\beta}), \qquad f(\xi)=\left\{\begin{array}{ll}f_1(\xi;A), & \xi\in[0,\overline{\xi}],\\ f_2(\xi), &\xi\geq\overline{\xi}, \end{array}\right.
\end{equation}
and notice that this is indeed a supersolution for any $T>0$, as it can be described alternatively as
$$
z(x,t)=\min\{z_1(x,t),z_2(x,t)\}, \qquad z_i(x,t)=(T-t)^{-\alpha}f_i((1+|x|)(T-t)^{\beta}), \qquad i=1,2,
$$
and $z_1$, $z_2$ are in fact solutions. The same considerations about the magnitude of $\sigma$ as in the end of Step 2 ensure that the supersolution defined in \eqref{supersol.N} works also for values of $\sigma$ smaller than $2(1-p)/(m-1)$. We thus complete the proof by choosing a sufficiently small $T_0>0$ such that \eqref{interm11} holds true on the support of $u_0$, and notice that $z$ is a supersolution to the approximating problems $(P_k)$ leading to the minimal solution. This again implies that $T_{\infty}\geq T_0>0$, completing the proof.
\end{proof}
The solution $\underline{u}$ constructed in the proof of Proposition \ref{prop.minimal} will be referred as \emph{the minimal solution} to the Cauchy problem \eqref{eq1}-\eqref{eq2} and denoted by $M(u_0)$ in the sequel. Notice that, the above proof does not imply that necessarily the minimal solution blows up in finite time. In fact, it might blow up or not according to whether $\sigma$ is larger or smaller than $2(1-p)/(m-1)$, but this is not easy to prove once we miss a comparison principle, and we refer the reader to the section of open problems at the end.

\section{Non-uniqueness for $m+p\geq2$}\label{sec.nonuniq}

A natural question raised by the previous section is whether in the case $m+p\geq2$ and for compactly supported and continuous data $u_0$ the minimal solution $\underline{u}$ constructed in Proposition \ref{prop.minimal} is the only solution to the Cauchy problem \eqref{eq1}-\eqref{eq2}. For $\sigma=0$ non-uniqueness follows easily from the construction of a different solution called maximal, which is shown to be strictly positive for any $t>0$ and thus different from $\underline{u}$, see \cite{dPV91}. We cannot construct such a maximal solution to Eq. \eqref{eq1}, since strictly positive solutions might not even exist at all in some ranges of $\sigma$ (see a comment with a formal intuition for such non-existence in the final Section of this work). In order thus to prove the existence of multiple solutions (and in fact an infinite number of them) we adapt to our case the deeper results in \cite{dPV92}, where infinitely many solutions to the Cauchy problem \eqref{eq4}-\eqref{eq2} are constructed, based on a prescribed evolution of the interface of them. As a preliminary fact, let us recall that Eq. \eqref{eq4} (that is, the case $\sigma=0$) admits an \emph{absolute minimal solution} in self-similar form
\begin{equation}\label{abs.min}
E(x,t)=t^{1/(1-p)}\varphi(|x|t^{-\gamma}), \qquad \gamma=\frac{m-p}{2(1-p)}
\end{equation}
with zero initial condition $E(x,0)=0$ for any $x\in\real$, according to \cite{dPV90}. It is also shown in \cite{dPV90} that such solution lies below any solution (and supersolution) to Eq. \eqref{eq4} and consequently also below any solution to Eq. \eqref{eq1} (which is a strict supersolution to Eq. \eqref{eq4}). Moreover, the profile $\varphi$ of $E$ is non-increasing and compactly supported, thus ${\rm supp}\,E\subseteq[-\varrho_0t^{\gamma},\varrho_0t^{\gamma}]$ where $\varrho_0>0$ is the right-interface point of the profile $\varphi$.

With these elements and notation in mind, we can prove Theorem \ref{th.nonuniq} as an immediate consequence of a stronger result adapting a construction from \cite{dPV92}. More precisely, by prescribing the behavior of the interface (under some limitations) we can obtain a solution to the Cauchy problem \eqref{eq1}-\eqref{eq2} with exactly that given interface at every (small) time $t\in(0,T)$. We formalize this below in dimension $N=1$, but the analogous result for radially symmetric solutions in dimension $N\geq2$ will be then completely analogous.
\begin{proposition}\label{prop.nonuniq1}
Let $u_0\in C_0(\real)$ be a compactly supported initial condition such that ${\rm supp}\,u_0=[r_0,R_0]$ for some $r_0$, $R_0\in\real$. Let $M(u_0)$ be the minimal solution to the Cauchy problem \eqref{eq1}-\eqref{eq2} with initial condition $u_0$ defined for $t\in(0,T_0)$ and denote by $s_l(t)$, $s_r(t)$ the left and right interfaces of $M(u_0)(t)$ for $t\in(0,T_0)$, that is, ${\rm supp}\,M(u_0)(t)=[s_l(t),s_r(t)]$. Let $\xi_l(t)$, $\xi_r(t)$ be two continuous functions of time such that $\xi_l(0)\leq r_0$, $\xi_r(0)\geq R_0$ and
$$
\xi_l(t_1)-\xi_l(t_2)\geq s_l(t_1)-s_l(t_2), \qquad \xi_r(t_2)-\xi_r(t_1)\geq s_r(t_2)-s_r(t_1),
$$
for any $t_1$, $t_2\in(0,T_0)$ such that $t_1<t_2$. Then there exists at least a shorter time interval $(0,T)\subset(0,T_0)$ and a solution $u$ to the Cauchy problem \eqref{eq1}-\eqref{eq2} with initial condition $u_0$ defined for $t\in(0,T)$ such that its left and right interfaces are given exactly by $\xi_l(t)$ and $\xi_r(t)$ for any $t\in(0,T)$.
\end{proposition}
\begin{proof}
We divide the proof into two steps.

\medskip

\noindent \textbf{Step 1. The fundamental extension.} This step adapts to our problem Step 1 in the proof of \cite[Proof of Corollary 1.2]{dPV92} and at its end in fact we already have the proof of Theorem \ref{th.nonuniq}. The goal here is to construct a solution to our Cauchy problem whose support has an instantaneous jump to the right from $R_0=s_r(0)$ to $R_0+r$ at time $t=0$ for a given (fixed) $r>0$ (and of course, a perfectly similar construction can be done for the left interface). Let then $r>0$ be given. For any $\epsilon>0$ sufficiently small (it is enough to start for example from $\epsilon<\|u_0\|_{\infty}/2$) there exists a last, closest point (that we denote by $R_{\epsilon}$) to the right interface $R_0$ of $u_0$ such that $u_0(R_{\epsilon})=\epsilon$. Let us introduce the following compactly supported and continuous initial condition
\begin{equation}\label{new.data}
u_{\epsilon,0}(x)=\left\{\begin{array}{ll}u_0(x), & {\rm for} \ x\in(-\infty,R_{\epsilon}),\\ \frac{\epsilon(R_0+r-x)}{R_0+r-R_{\epsilon}}, & \rm{for} \ x\in[R_{\epsilon},R_0+r], \\ 0, & {\rm for} \ x\in(R_0+r,\infty),\end{array}\right.
\end{equation}
that is, adding a linear extension to $u_0$ from $R_{\epsilon}$ to the new edge of the support $R_0+r$. Let $u_{\epsilon}=M(u_{\epsilon,0})$ be the minimal solution associated to this new condition. By comparison between minimal solutions (which holds true since it is obvious that any solution to Eq. \eqref{eq1} with a larger initial condition than the given $u_0$ becomes a supersolution to any of the problems $(P_k)$ approximating the minimal solution $M(u_0)$) it readily follows that $u_{\epsilon_2}(x,t)\geq u_{\epsilon_1}(x,t)$ for any $\epsilon_2>\epsilon_1>0$ and at any $(x,t)\in\real\times(0,T)$ where $T>0$ is for example the lifetime of the solution with the biggest $\epsilon$ chosen. Recalling the construction of the absolute minimal solution $E(x,t)$ to Eq. \eqref{eq4} (see the details in \cite[Theorem 4]{dPV90}) and the fact that any non-trivial solution to our Eq. \eqref{eq1} is a strict supersolution to Eq. \eqref{eq4} we easily conclude that for any $(x_0,t_0)\in\real\times(0,T)$ with $x_0\in[R_{\epsilon},R_0+r]$ we have
\begin{equation}\label{interm8}
u_{\epsilon}(x,t)\geq E(x-x_0,t-t_0), \qquad t>t_0.
\end{equation}
We can then define
$$
\overline{u}(x,t)=\lim\limits_{\epsilon\to0}u_{\epsilon}(x,t), \qquad (x,t)\in\real\times(0,T),
$$
which is a monotone limit and it is easy to see that $\overline{u}$ is a weak solution to Eq. \eqref{eq1} using the Monotone Convergence Theorem (in fact we even have uniform convergence by Dini's Theorem since the limit is continuous). Let us stress here that we do not need in this construction a bound from above to guarantee that the limit is finite, since we are dealing with a decreasing limit. It is then obvious that
$$
\overline{u}(x,0)=\lim\limits_{\epsilon\to0}u_{\epsilon,0}(x)=u_0(x)
$$
and the comparison from above with $E$ in \eqref{interm8} transfers to the limit $\overline{u}$, proving that for any $t>0$ we have $\overline{u}(x,t)>0$ for $x\in(R_0,R_0+r)$.

\medskip

\noindent \textbf{Step 2. The iterative construction.} We perform now an iterative construction based on a discretization in time and the application of Step 1 in each interval of the discretization to perform an instantaneous jump in the supports. Let then $\xi_l(t)$ and $\xi_r(t)$ be two continuous, increasing functions of time as in the statement of Proposition \ref{prop.nonuniq1}. Fix a time $t\in(0,T)$ for a $T>0$ sufficiently small (that will be chosen later). For any positive integer $n$ we construct an approximate solution $u_n$ as in \cite[Proof of Corollary 1.2]{dPV92}. We briefly and sketchy describe the construction here for the sake of completeness. Consider a partition
$$
0=t_0<t_1<t_2<...<t_{n-1}<t_n=t, \qquad t_j=\frac{jt}{n}
$$
and construct the function $u_n$ by induction. More precisely, assume first that $u_n$ is already constructed for $t\in[0,t_j]$ and we want to pass to $t_{j+1}$. To this end, we begin from the edges of the support of $u_n(t_j)$ and we perform a jump of them by applying Step 1 both to the right (with $r=\xi_r(t_{j+1})-\xi_r(t_j)$) and to the left (with $r=\xi_l(t_j)-\xi_l(t_{j+1})$), using here the fact that the speed of advance of the prescribed interfaces to the left and to the right is higher than the ones of the minimal solution with initial condition $u_0$. The precise details are easy and given in the above mentioned proof of Corollary 1.2 in \cite{dPV92}. In order to pass to the limit as $n\to\infty$ in the iteration we need an uniform bound from above to show that the limit solution does not escape to infinity. We cannot use a translation of a minimal solution or a construction based on it (as it was done in \cite[Lemma 2.4]{dPV92}) since our equation is not invariant to translations, but instead we can bound uniformly from above the iterated solutions by a sufficiently big non-increasing supersolution in self-similar form
$$
U(x,t)=(T-t)^{-\alpha}f((1+|x|)(T-t)^{\beta}), \ \alpha=\frac{\sigma+2}{L}, \ \beta=\frac{m-p}{L}
$$
similar to the ones introduced in \eqref{SSS} with a profile $f(\xi)$ solving \eqref{SSODE} and having a right interface at $\xi_0\in(0,\infty)$. Indeed, comparison from above with such a supersolution $U$ can be performed as the iterative construction of $u_n$ is based on adding up at each iteration step only minimal solutions, provided that at our fixed time $t>0$ for which we have built $u_n$ we have ordered supports between $u_n$ and $U$, that is
\begin{equation}\label{interm9}
\xi_r(t)<\xi_0(T-t)^{-\beta}, \qquad \xi_l(t)>-\xi_0(T-t)^{-\beta},
\end{equation}
which also gives a limitation for the lifetime $T>0$ depending on the two prescribed functions $\xi_r$, $\xi_l$. We notice that for ``faster" advancing interface functions $\xi_l(t)$, $\xi_r(t)$, smaller lifetime $T$ is expected. Once satisfied this condition, we can pass to the limit in the discretization as $n\to\infty$ and obtain the desired weak solution
$$
v(x,t)=\lim\limits_{n\to\infty}u_n(x,t), \qquad t\in(0,T),
$$
with $T>0$ chosen sufficiently small according to \eqref{interm9}. The bound from below by $E$ at every point of positivity (as done in Step 1) together with the construction and the continuity of the prescribed functions $\xi_l$, $\xi_r$ ensure that the left and right interfaces of $v$ at every time $t\in(0,T)$ are given exactly by the continuous functions $\xi_l(t)$ and $\xi_r(t)$. In the meantime, the universal bound from above by the supersolution $U$ in self-similar form ensures that at least in the time interval $(0,T)$ we have $v(x,t)<\infty$ for any $x\in\real$.
\end{proof}

\medskip

\noindent \textbf{Remarks. 1.} Step 1 above already completes the proof of Theorem \ref{th.nonuniq} in dimension $N=1$. Indeed, for every $r>0$ we can construct a different solution to the same Cauchy problem \eqref{eq1}-\eqref{eq2}. A totally analogous extension can be constructed for radially symmetric initial conditions $u_0\in C_0(\real^N)$ by replacing $x$ by $r=|x|$, which allows us to work with the radial variable exactly as in dimension $N=1$ but using for comparison from above the supersolutions introduced in Step 2 of the proof of Proposition \ref{prop.minimal}, thus completing the proof of Theorem \ref{th.nonuniq} also in dimension $N\geq2$.

\noindent \textbf{2.} Our Proposition \ref{prop.nonuniq1} also holds true for $\sigma=0$ and improves slightly the result for Eq. \eqref{eq4} in \cite[Theorem 1.1 and Theorem 5.1]{dPV92} since we do not need to consider the non-increasing majorant $\tilde{u_0}$ of the initial condition as considered in the above mentioned work.

\section{Aronson-B\'enilan estimates when $m+p<2$}\label{sec.ABE}

This section is devoted to the deduction of the Aronson-B\'enilan estimates \eqref{ABE} in the homogeneous case $\sigma=0$. Let us recall here the \emph{pressure function}
$$
v=\frac{m}{m-1}u^{m-1},
$$
since most of the forthcoming work will be performed on this function.
\begin{proof}[Proof of Theorem \ref{th.ABE}]
The proof is inspired from the one for \cite[Proposition 9.4]{VPME} but technically more involved. We first derive after rather straightforward calculations the \emph{pressure equation} which will be used further in the present work, that is, the parabolic PDE satisfied by the function $v$ introduced above:
\begin{equation}\label{eq.pres}
\begin{split}
&v_t=(m-1)v\Delta v+|\nabla v|^2+K(m,p)v^{(m+p-2)/(m-1)}, \\ &K(m,p)=m\left(\frac{m-1}{m}\right)^{(m+p-2)/(m-1)}.
\end{split}
\end{equation}
Let us notice here the strong influence of the sign of $m+p-2$ on this equation, due to its last term. In order to go further, we set $w=\Delta v$ and we next derive the partial differential equation solved by $w$. To this end, we calculate the terms separately. On the one hand, for the reaction term we get
\begin{equation*}
\frac{\partial}{\partial x_i}K(m,p)v^{(m+p-2)/(m-1)}=K(m,p)\left[\frac{m+p-2}{m-1}v^{(p-1)/(m-1)}\frac{\partial v}{\partial x_i}\right]
\end{equation*}
hence
\begin{equation*}
\frac{\partial^2}{\partial x_i^2}K(m,p)v^{\frac{m+p-2}{m-1}}=K(m,p)\left[\frac{m+p-2}{m-1}v^{\frac{p-1}{m-1}}\frac{\partial^2v}{\partial x_i^2}+\frac{(2-m-p)(1-p)}{(m-1)^2}v^{\frac{p-m}{m-1}}\left(\frac{\partial v}{\partial x_i}\right)^2\right].
\end{equation*}
The contribution of the reaction term thus gives
\begin{equation*}
\begin{split}
\sum\limits_{i=1}^N&\frac{\partial^2}{\partial x_i^2}K(m,p)v^{\frac{m+p-2}{m-1}}=K(m,p)\frac{m+p-2}{m-1}v^{\frac{p-1}{m-1}}w\\&+K(m,p)\frac{(2-m-p)(1-p)}{(m-1)^2}v^{\frac{p-m}{m-1}}|\nabla v|^2=\mathcal{R}_1(x,v)w+\mathcal{R}_2(x,v),
\end{split}
\end{equation*}
where
\begin{equation*}
\begin{split}
&\mathcal{R}_1(x,v)=K(m,p)\frac{m+p-2}{m-1}v^{\frac{p-1}{m-1}}<0, \\ &\mathcal{R}_2(x,v)=K(m,p)\frac{(2-m-p)(1-p)}{(m-1)^2}v^{\frac{p-m}{m-1}}|\nabla v|^2\geq0.
\end{split}
\end{equation*}
On the other hand, the diffusion term can be worked out as in \cite[Proposition 9.4]{VPME} to get in the end
\begin{equation*}
\begin{split}
w_t&=(m-1)v\Delta w+2m\nabla v\cdot\nabla w+(m-1)w^2+2\sum\limits_{i,j=1}^N\left(\frac{\partial^2 v}{\partial x_i\partial x_j}\right)^2+\mathcal{R}_1(x,v)w+\mathcal{R}_2(x,v)\\
&\geq(m-1)v\Delta w+2m\nabla v\cdot\nabla w+\left(m-1+\frac{2}{N}\right)w^2+\mathcal{R}_1(x,v)w+\mathcal{R}_2(x,v)
\end{split}
\end{equation*}
after a straightforward application of the Cauchy-Schwartz inequality. This can be written equivalently as $\mathcal{L}w\geq\mathcal{R}_2(x,v)$, where
$$
\mathcal{L}w:=w_t-(m-1)v\Delta w-2m\nabla v\cdot\nabla w-\left(m-1+\frac{2}{N}\right)w^2-\mathcal{R}_1(x,v)w
$$
is a uniformly parabolic operator. Since $\mathcal{R}_2(x,v)\geq0$, we deduce that $\mathcal{L}w\geq0$ and we aim to find a subsolution for $\mathcal{L}$ depending only on time. We thus take for $t>0$
$$
\underline{W}(x,t)=-\frac{C}{t}, \qquad C=\frac{N}{N(m-1)+2}
$$
and calculate
$$
\mathcal{L}\underline{W}=\frac{C}{t^2}-\frac{C}{t^2}+\frac{C\mathcal{R}_1(x,v)}{t}<0,
$$
since $m+p-2<0$ in our range of exponents. Applying the comparison principle to the parabolic operator $\mathcal{L}$ we infer that
$$
\Delta v(x,t)=w(x,t)\geq\underline{W}(t)=-\frac{N}{(N(m-1)+2)t}, \qquad {\rm for \ any} \ (x,t)\in\real^N\times(0,\infty)
$$
as stated.

\medskip

All the previous calculations and the application of the maximum principle to the operator $\mathcal{L}$ are fully justified for solutions $u$ such that (in the pressure variable) $\mathcal{L}$ is uniformly parabolic, that is, when $v$, $\nabla v$ are bounded and $v>0$ uniformly. In order to extend the Aronson-B\'enilan estimates to general weak solution we proceed by approximation. Let us first consider $u_0$ to be an initial condition as in \eqref{icond} and continuous. Then, there exists a unique solution $u$ to the Cauchy problem \eqref{eq4}-\eqref{eq2} in a time interval $(0,T)$, according to \cite{dPV91}. Let $u_k$ be the solution to the Cauchy problem with initial condition
$$
u_{0,k}(x)=u_0(x)+\frac{1}{k},
$$
for any positive integer $k$. We infer from \cite[Theorem 2.1]{dPV91} and its proof that there exists a unique solution $u_k$ to \eqref{eq4} with initial condition $u_{0,k}$, and it satisfies $u_k(x,t)\geq 1/k$ for any $x\in\real^N$, $t>0$. We further find from \cite[Theorem 8.1, Chapter V]{LSU} (which applies for our approximating solutions $u_k$ since they are now bounded from below by a positive constant) that $u_k$ has the regularity required for \eqref{eq4} to hold true in a classical sense and all the space derivatives of $u_k$ up to the second order and time derivatives of $u_k$ up to order one are uniformly bounded. In this case, the previous calculation applies rigorously for $u_k$ and we get that
\begin{equation}\label{interm12}
\Delta v_k(x,t)\geq-\frac{K}{t}, \qquad v_k=\frac{m}{m-1}u_k^{m-1}.
\end{equation}
Moreover, the comparison principle for \eqref{eq4} (see \cite[Theorem 2.1]{dPV91}) entails that solutions $u_k$ for $k\geq1$ form a non-increasing sequence of functions, thus there exists a limit
$$
\overline{u}(x,t)=\lim\limits_{k\to\infty}u_k(x,t), \qquad (x,t)\in\real^N\times(0,T),
$$
and the uniform bound of $u_k$ and their derivatives up to second order together with the Arzel\'a-Ascoli Theorem imply that $u_k\to\overline{u}$ locally uniformly and the same holds true for their first order derivatives with respect to the space variables. Then, the uniform boundedness of $u_k$ and $\partial_tu_k$ gives the continuity with respect to the time variable over $(0,T)$ of the limit function $\overline{u}$, while the fact that $\overline{u}$ belongs to $L^{\infty}_{\rm loc}$ is obvious, as it is bounded from above by any $u_k$. We thus fulfill the regularity assumption (a) in Definition \ref{def.sol}. The monotone convergence theorem then easily gives that $\overline{u}$ satisfies assumptions (b) and (c) in Definition \ref{def.sol}, hence, since $u_k(x,0)=u_{0,k}(x)\to u_0(x)$ as $k\to\infty$ on $\real^N$, we readily infer that $\overline{u}$ is a weak solution to the Cauchy problem \eqref{eq4}-\eqref{eq2}. Uniqueness of solutions to the latter Cauchy problem, established in \cite{dPV90} for continuous and bounded initial conditions, then proves that $\overline{u}=u$. We then come back to \eqref{interm12}, which, after multiplication by a non-negative test function and integration by parts, writes
\begin{equation}\label{interm13}
\int_0^{T}\int_{\real^N}\left(v_k(x,t)\Delta\varphi(x,t)+\frac{K}{t}\varphi(x,t)\right)\,dx\,dt\geq0,
\end{equation}
for any $\varphi\in C_0^{\infty}(\real^N\times(0,T))$, $\varphi\geq0$. We pass to the limit in \eqref{interm13} as $k\to\infty$, taking into account that $v_k\to v$ locally uniformly as $k\to\infty$, and obtain the claimed distributional form \eqref{ABEdist}.
\end{proof}
We end this section with a corollary which will be used in the sequel.
\begin{corollary}\label{cor.ABE}
In the same conditions as in Theorem \ref{th.ABE}, we have
$$
u_t\geq-\frac{Ku}{t}, \qquad K=\frac{N}{N(m-1)+2},
$$
in the sense of distributions in $\real^N$.
\end{corollary}
\begin{proof}
At a formal level, we infer from \eqref{eq.pres} that $v_t\geq(m-1)v\Delta v$, hence
$$
(m-1)\frac{u_t}{u}=\frac{v_t}{v}\geq(m-1)\Delta v\geq-\frac{N(m-1)}{(N(m-1)+2)t}
$$
and we reach the conclusion with the same constant $K$ as in \eqref{ABE}. For general weak solutions the estimate is proved by using the same approximation as in the proof of Theorem \ref{th.ABE}.
\end{proof}

\medskip

\noindent \textbf{Remark. Formal proof of Aronson-B\'enilan estimates for} $\mathbf{\sigma>0}$. At a formal level, the Aronson-B\'enilan estimates \eqref{ABE} or \eqref{ABEdist} hold true also for $\sigma>0$ and $N\geq2$. Indeed, a slightly longer but straightforward calculation along the first few lines of the proof of Theorem \ref{th.ABE} by computing the derivatives up to second order of the reaction term, but applied to Eq. \eqref{eq1} with $\sigma>0$, gives
\begin{equation*}
\begin{split}
\sum\limits_{i=1}^N&\frac{\partial^2}{\partial x_i^2}K(m,p)(1+|x|)^{\sigma}v^{\frac{m+p-2}{m-1}}=K(m,p)\frac{m+p-2}{m-1}(1+|x|)^{\sigma}v^{\frac{p-1}{m-1}}w\\
&+K(m,p)(1+|x|)^{\sigma-2}v^{\frac{p-m}{m-1}}\left[\frac{(2-m-p)(1-p)}{(m-1)^2}(1+|x|)^2|\nabla v|^2\right.\\
&\left.-2\sigma\frac{2-m-p}{m-1}(1+|x|)v\frac{x}{|x|}\cdot\nabla v+\sigma\left(\sigma-1+(N-1)\frac{1+|x|}{|x|}\right)v^2\right]\\
&=\mathcal{R}_1(x,v)w+\mathcal{R}_2(x,v),
\end{split}
\end{equation*}
where
$$
\mathcal{R}_1(x,v)=K(m,p)\frac{m+p-2}{m-1}(1+|x|)^{\sigma}v^{\frac{p-1}{m-1}}<0
$$
and $\mathcal{R}_2(x,v)$ gathers the rest of the terms. In order to proceed with the comparison principle as we did in the body of the proof of Theorem \ref{th.ABE}, we still need to have $\mathcal{R}_2(x,v)\geq0$. To this end, we write $\mathcal{R}_2$ as a square and we examine the remainders. More precisely, using once more a standard Cauchy-Schwarz inequality for the scalar product $\nabla v\cdot x/|x|$ we find
\begin{equation*}
\begin{split}
\frac{\mathcal{R}_2(x,v)}{K(m,p)(1+|x|)^{\sigma-2}v^{\frac{p-m}{m-1}}}&\geq\left[\frac{2-m-p}{m-1}(1+|x|)|\nabla v|-\sigma v\right]^2+\frac{2-m-p}{m-1}(1+|x|)^2|\nabla v|^2\\
&+\sigma\left[(N-1)\frac{1+|x|}{|x|}-1\right]v^2
\end{split}
\end{equation*}
and taking into account that we are in the range $m+p-2<0$, it follows that $\mathcal{R}_2(x,v)\geq0$ provided that
$$
\sigma\left[(N-1)\frac{1+|x|}{|x|}-1\right]\geq0,
$$
which holds true for $\sigma>0$ if $N\geq2$. We thus conclude, at a formal level, that \eqref{ABEdist} should hold true in this case. However, we left this part out of the statement of Theorem \ref{th.ABE} since the final approximation argument leading to the rigorous proof for $\sigma=0$ cannot be performed, as we are missing an existence and uniqueness result for solutions to the Cauchy problem \eqref{eq1}-\eqref{eq2} when $\sigma>0$.

\section{Infinite speed of propagation when $m+p<2$}\label{sec.ISP}

In this part we use the Aronson-B\'enilan estimates in Theorem \ref{th.ABE} to establish the infinite speed of propagation of the supports of solutions to Eq. \eqref{eq1} when $m+p-2<0$ and thus complete the proof of Theorem \ref{th.ISP}. Let us stress again here that Theorem \ref{th.ISP} has been proved in \cite[Lemma 2.4]{dPV90} for $\sigma=0$. We give here an independent proof, based on a completely different argument, and extend it to exponents $\sigma>0$.
\begin{proof}[Proof of Theorem \ref{th.ISP}]
In a first step, let $\sigma=0$ and assume for contradiction that for some compactly supported initial condition $u_0$ as in \eqref{icond}, $u(t)$ remains compactly supported for $t\in(0,T_0)$. Recall the pressure equation \eqref{eq.pres}, which for $\sigma=0$ writes
\begin{equation}\label{eq.pres2}
v_t=(m-1)v\Delta v+|\nabla v|^2+K(m,p)v^{(m+p-2)/(m-1)}.
\end{equation}
At a formal level, since $m+p-2<0$, fixing some $t\in(0,T_0)$ one reaches easily a contradiction in \eqref{eq.pres2}. Indeed, since $\Delta v(x,t)\geq-K/t$ and $|\nabla v(x,t)|^2\geq0$ at any point $x\in\real^N$, it follows that at the interface point $x=s(t)$ we get $v_t(s(t),t)=+\infty$ in order to compensate the negative power $(m+p-2)/(m-1)$ in the last term of the right hand side and for any $t\in(0,T_0)$. This is obviously equivalent to the infinite speed of propagation of the supports. More rigorously, since $m+p<2$ we multiply by $v^{(2-m-p)/(m-1)}$ in \eqref{eq.pres2} and also by a test function $\varphi\in C_0^{\infty}(\real^N)$, $\varphi\geq0$, then we integrate on $\real^N$ and on any time interval $(\tau_0,\tau_1)\subset(0,T_0)$ and we drop the second term in the right hand side (which is always positive) to obtain
\begin{equation}\label{interm10}
\begin{split}
\frac{m-1}{1-p}&\int_{\tau_0}^{\tau_1}\int_{\real^N}(v^{(1-p)/(m-1)})_t\varphi\,dx\,dt\geq(m-1)\int_{\tau_0}^{\tau_1}\int_{\real^N}v^{(1-p)/(m-1)}\Delta v\varphi\,dx\,dt\\&+K(m,p)\int_{\tau_0}^{\tau_1}\int_{\real^N}\varphi\,dx\,dt\\
&\geq-\frac{N(m-1)}{N(m-1)+2}\int_{\tau_0}^{\tau_1}\int_{\real^N}\frac{1}{t}v^{(1-p)/(m-1)}\varphi\,dx\,dt+K(m,p)\int_{\tau_0}^{\tau_1}\int_{\real^N}\varphi\,dx\,dt.
\end{split}
\end{equation}
We now consider a sequence of test functions $(\varphi_n)_{n\geq1}$ defined as follows
$$
\varphi_n(x)=1, \ {\rm for} \ |x|\leq n, \quad 0\leq\varphi_n(x)\leq 1 \ {\rm for \ any} \ x\in\real^N, \quad {\rm supp}\varphi_n\subseteq B(0,2n),
$$
where $B(0,2n)=\{x\in\real^N: |x|\leq 2n\}$, and let $\varphi=\varphi_n$ in \eqref{interm10} for any positive integer $n\geq1$. Since the support of $v$ is uniformly localized for $t\in[\tau_0,\tau_1]$ (as $\tau_1<T_0$), it follows that the right-hand side of \eqref{interm10} tends to $+\infty$ as $n\to\infty$ due to the last integral only of $\varphi_n$. It thus follows that
$$
\lim\limits_{n\to\infty}\int_{\tau_0}^{\tau_1}\int_{\real^N}(v^{(1-p)/(m-1)})_t\varphi_n\,dx\,dt=+\infty
$$
or equivalently
$$
\lim\limits_{n\to\infty}\int_{\real^N}\left[v^{(1-p)/(m-1)}(\tau_1)-v^{(1-p)/(m-1)}(\tau_0)\right]\varphi_n\,dx=+\infty,
$$
which is a contradiction with the localization of the supports of $v(t)$ for $t\in[\tau_0,\tau_1]$. Let us notice here that the chosen range of exponents $p\in(0,1)$ and $m+p<2$ was decisive, as after multiplication by a positive power $v^{(2-m-p)/(m-1)}$, we got in the left hand side also a positive power $v^{(1-p)/(m-1)}$ of $v$, thus we do not create new singularities at the edges of the supports.

We pass now to $\sigma>0$. Assume that there is a weak solution $u$ to the Cauchy problem Eq. \eqref{eq1}-\eqref{eq2} defined for $t\in(0,T)$ with some $T>0$. Since $(1+|x|)^{\sigma}\geq1$ for any $x\in\real^N$ we deduce that $u$ is a supersolution to the Cauchy problem \eqref{eq4}-\eqref{eq2}. We infer from the comparison principle (which holds true for \eqref{eq4} and non-trivial initial data in the range $m+p<2$, \cite{dPV91}) that $u(x,t)>0$ for any $(x,t)\in\real^N\times(0,T)$.
\end{proof}

\section*{Some extensions and open problems}

We gather here some extensions related to the previous results, that we consider interesting.

\medskip

\noindent \textbf{1. Finite time blow-up.} A natural question is whether any solution to Eq. \eqref{eq1} blows up in finite time or there are some initial conditions $u_0$ producing (minimal) solutions that are global in time. Our conjecture is that, if $L>0$, \emph{any non-trivial solution is expected to blow up} in finite time, while if $L\leq0$, there are initial conditions producing solutions that are global in time (we recall that $L$ is defined in \eqref{const}). A formal argument about general finite time blow-up if $L>0$ is based on comparison with subsolutions in self-similar form obtained in our recent papers \cite{IS20b, IS20c, IMS22}. More precisely, it goes by contradiction as follows: assume that there exists $u_0\in C_0(\real^N)$ as in \eqref{icond} such that the minimal solution $M(u_0)$ is defined for $t\in(0,\infty)$. We infer by comparison with the absolute minimal solution $E$ defined in \cite{dPV90} and \eqref{abs.min} that for large $t$, solution $M(u_0)(t)$ is as large as we want both in amplitude and support. We can thus find a blow-up self-similar solution to Eq. \eqref{eq3} as in \cite{IS20b, IS20c, IMS22} (which is a subsolution to Eq. \eqref{eq1}) below it. If comparison would be allowed, then we would get an easy contradiction proving that any minimal solution (and then any other solution) blows up in finite time.

But as we see very well by considering the subsolution $\underline{U}$ with initial condition $\underline{U}(x,0)=0$ for any $x\in\real^N$, which is defined explicitly by
\begin{equation}\label{subsol}
\underline{U}(x,t)=\left(\frac{1}{1-p}\right)^{1/(p-1)}t^{1/(1-p)}(1+|x|)^{\sigma/(1-p)},
\end{equation}
comparison does not hold true in general, as otherwise any solution (even the ones with compact support) could have been compared to $\underline{U}$ in order to force it to become positive everywhere, contradicting the finite speed of propagation at least in the range $m+p\geq2$. This opens the question on whether the self-similar solutions to Eq. \eqref{eq1} are minimal solutions in the sense of the construction performed in Proposition \ref{prop.minimal}.

\medskip

\noindent \textbf{2. Non-existence of positive solutions if $\sigma>2(1-p)/(m-1)$}. Strongly connected to the first comment, and expecting (at a formal level) that we might compare a strictly positive solution to Eq. \eqref{eq1} with the subsolution $\underline{U}$ introduced in \eqref{subsol}, we conjecture that, if $L>0$ (that is, $\sigma>2(1-p)/(m-1)$) there are no solutions to Eq. \eqref{eq1} such that $u(x,t)>0$ for any $x\in\real^N$. Indeed, assuming that the comparison can be performed rigorously if $u(x,t)>0$ for any $x\in\real^N$, we would get a solution with local behavior
$$
u(x,t)\geq C(1+|x|)^{\sigma/(1-p)}, \qquad {\rm as} \ |x|\to\infty, \qquad {\rm for \ any} \ t>0.
$$
But any solution to Eq. \eqref{eq1} is a supersolution to the standard porous medium equation and classical results on the porous medium equation (see for example \cite{AC83, BCP84, CVW87}) state that there are no solutions (and it seems to us that the proofs can be extended to supersolutions) to it increasing at infinity faster than $|x|^{2/(m-1)}$. Since $\sigma/(1-p)>2/(m-1)$ if $L>0$, we would be in this case. In particular, since infinite speed of propagation is in force for $m+p<2$ by Theorem \ref{th.ISP}, we expect complete non-existence of non-trivial solutions if $L>0$ and $m+p<2$.

\medskip

\noindent \textbf{3. Establishing which self-similar solutions are minimal.} An immediate adaptation of the result in \cite{IS20b} proves that Eq. \eqref{eq1} with $m+p>2$ presents two types of blow-up self-similar solutions
$$
U(x,t)=(T-t)^{-\alpha}f((1+|x|)(T-t)^{\beta}), \qquad \alpha=\frac{\sigma+2}{L}, \ \beta=\frac{m-p}{L}
$$
which differ with respect to the local behavior of the profile $f(\xi)$ near the interface point $\xi_0\in(0,\infty)$, namely
$$
f(\xi)\sim(\xi_0-\xi)^{1/(m-1)} \ ({\rm Type \ I}) \ \ {\rm or} \ \ f(\xi)\sim(\xi_0-\xi)^{1/(1-p)} \ ({\rm Type \ II}),
$$
both taken in the limit $\xi\to\xi_0$, $\xi<\xi_0$. It is also shown at least at a formal level that these solutions satisfy two different interface equations. Thus, a natural question would be whether any of these self-similar solutions is minimal in the sense of the construction in Proposition \ref{prop.minimal}, which would allow us to compare and conclude on the finite time blow-up. This is not an easy question and our intuition suggests that minimality has to do with the interface equation: we might expect that the solutions with interface of Type I are minimal, while the other ones are not. This conjecture is supported by the analogy with the minimality of the traveling wave solutions to Eq. \eqref{eq4}, see for example \cite[Theorem 4.1]{dPV91}.

\medskip

\noindent \textbf{4. Connection between non-uniqueness and blow-up time.} A much deeper open question is related to whether, in the range $\sigma>2(1-p)/(m-1)$, prescribing a blow-up time $T$ and a function $u_0$, there exists a unique solution to the Cauchy problem \eqref{eq1}-\eqref{eq2} with condition $u_0$ blowing up in finite time exactly at the given time $T$. More precisely, taking $u_0\in C_0(\real)$ satisfying \eqref{icond}, there exists a minimal solution $M(u_0)$ which (assuming that point 1 in this enumeration of open problems holds true, as we strongly expect) comes with a finite blow-up time $T_0\in(0,\infty)$. We have also proved in Section \ref{sec.nonuniq} that the Cauchy problem \eqref{eq1}-\eqref{eq2} has an infinite number of compactly supported solutions with interfaces advancing faster than the minimal one, and estimate \eqref{interm9} shows that faster advancing speed of the interface implies shorter lifetime before blow-up. Thus one can wonder naturally whether, given $T\in(0,T_0)$, there exists one solution to the Cauchy problem \eqref{eq1}-\eqref{eq2} blowing up exactly at this time $T$. We do not have by now a suggestion of how to approach this problem, but it is in our opinion an interesting and deep open question.

\bigskip

\noindent \textbf{Acknowledgements.} R. G. I. and A. S. are partially supported by the Project PID2020-115273GB-I00 and by the Grant RED2022-134301-T funded by MCIN/AEI/10.13039/ \\ 501100011033 (Spain).

\bibliographystyle{plain}

\end{document}